\documentclass[a4paper,12pt]{article}

\usepackage{amsmath}
\usepackage{amssymb}
\usepackage{amsthm}
\usepackage[dvipdfmx]{graphicx}
\usepackage{amscd}
\usepackage{comment}

\newcommand{\supp}{\mathop{\mathrm{supp}}\nolimits}%AAA:supp, Fix\UTF{0082}\UTF{00C8}\UTF{0082}\UTF{00C7}%

\newcommand{\Min}{\mathrm{Min}}
\newcommand{\Inf}{\mathrm{inf}}

\title{On groups whose actions on finite-dimensional CAT(0) spaces have global fixed points}
\author{Motoko Kato}
\begin{document}

\numberwithin{equation}{section}
\newtheorem{theorem}{Theorem}[section]
\newtheorem{proposition}[theorem]{Proposition}
\newtheorem{formula}{Formula}[section]
\newtheorem{lemma}[theorem]{Lemma}
\newtheorem{corollary}[theorem]{Corollary}
\newtheorem{remark}{Remark}[section]
\newtheorem{definition}{Definition}[section]
\newtheorem{example}[theorem]{Example}
\newtheorem{question}[theorem]{Question}

\maketitle

\begin{abstract}
We give a criterion for group elements to have fixed points with respect to a semi-simple action on a complete CAT(0) space of finite topological dimension.  
As an application, we show that Thompson's group $T$ and various generalizations of Thompson's group $V$ have global fixed points when they act semi-simply on finite-dimensional complete CAT(0) spaces,
while it is known that
$T$ and $V$ act properly on infinite-dimensional CAT(0) cube complexes. 
\end{abstract}

\section{Introduction}

A CAT(0) space is a geodesic metric space whose geodesic triangles are no fatter than triangles in the Euclidean plane with the same edge lengths.
The existence of reasonable actions on CAT(0) spaces is important information for a group, related to various properties of groups. 

A group $G$ is said to have {\it Serre's property $FA$} if every action of $G$ on a simplicial tree, which is a $1$-dimensional CAT(0) cube complex, has a global fixed point.
If $G$ has Serre's property $FA$ then $G$ does not split as an amalgamated product or an HNN extension.

For every $k\in \mathbb{N}$, a group $G$ is said to have {\it property $F\mathcal{A}_{k}$} if every isometric action of $G$ on a complete CAT(0) space of
topological dimension $k$ has a global fixed point (\cite{Farb}, \cite{Varghese}).
We say that $G$ has {\it property $F\mathcal{A}_{k}$ for semi-simple actions} if every semi-simple action of $G$ on a complete CAT(0) space of
topological dimension $k$ has a global fixed point.
Property $\mathrm{F}\mathcal{A}_{k}$ is related to representations of groups.
For example, if a group $G$ has $\mathrm{F}\mathcal{A}_{k-1}$ then $G$ is of integral $k$-representation type (\cite{Farb}).

For $n\geq 3$, $\mathrm{SL}_n(\mathbb{Z}[1/p])$ has property $\mathrm{F}\mathcal{A}_{n-2}$ for semi-simple actions, while it does not have property $\mathrm{F}\mathcal{A}_{n-1}$ (\cite{Farb}).
For $n\geq 4$ and $k<\mathrm{min}\{i \lfloor {n}/{(i+2)}\rfloor \mid 2\leq i\leq k+1\}$, $\mathrm{Aut}(F_n)$ has property $\mathrm{F}\mathcal{A}_{k}$ (\cite{Varghese}).
For every $g\in \mathbb{N}$, the mapping class group of a closed orientable surface of genus $g$ has property $\mathrm{F}\mathcal{A}_{g-1}$ for semi-simple actions (\cite{Bridson}).

We argue property $\mathrm{F}\mathcal{A}_{k}$ for semi-simple actions
for Thompson's groups $T$, $V$ and their generalizations.
$T$ and $V$ are known as the first examples of finitely presented infinite simple groups.

It is known that Thompson's groups $T$, $V$ and $F$
act properly on infinite-dimensional CAT(0) cube complexes (\cite{Farley3} for $F$ and \cite{Farley}).
Since $F$ admits a structure of an $\mathrm{HNN}$-extention, $F$ has an isometric action without a global fixed point on a tree.
On the othe hand, $F$, $T$ and $V$ cannot act properly on finite dimensional CAT(0) complexes, since they contain free abelian subgroups of infinite rank (\cite{CFP}).
By considering these facts, it is natural to ask whether $T$ and $V$ can act isometrically on finite-dimensional CAT(0) complexes, without global fixed points.

In this article, 
we give a criterion for group elements to have fixed points with respect to a semi-simle action on a complete CAT(0) space of finite topological dimension.  
In the following, for a group $G$ acting on a set $A$,
{\it the support of an element} $g\in G$ is $\supp(g)=\{a\in A\mid ga\neq a\} \subset A$.  
{\it The support of a subset} $L$ is $\supp(L)=\bigcup_{l\in L} \supp(l)\subset A$. 
For a topological space $X$,
$\dim(X)$ denotes the topological dimension of $X$.
We say $X$ is $k$-dimensional if $\dim(X)=k$.

\begin{theorem}\label{MainNew}
Let $G$ be a group acting faithfully on a set $A$. 
Let $s\in G$ be an element satisfying the following conditions:
there is a sequence of subgroups $\langle s\rangle = H_0<H_1<\cdots<H_k<H_{k+1}=G$ 
with elements $g_i\in H_{i+1}$ $(1\leq i\leq k)$ such that for every $1\leq i\leq k$,
 \begin{itemize}
 \item[$(i)$] every homomorphism from $H_i$ to $\mathbb{R}$ is trivial, and
 \item[$(ii)$] $g_i(\supp(H_i))\cap \supp(H_i)=\emptyset$ in $A$.
 \end{itemize}  
Then for every complete $k$-dimensional CAT(0) space $X$, for every semi-simple action of $G$ on $X$, $s$ has a fixed point. 
\end{theorem}

By using Theorem~\ref{MainNew} and the following theorem by Bridson, we may confirm that many Thompson-type groups have property $\mathrm{F}\mathcal{A}_{k}$ for semi-simple actions for every $k\in\mathbb{N}$.

\begin{theorem}[\cite{Bridson}, \cite{Varghese}]\label{Bridson}
Let $k_1,\ldots, k_r$ be positive integers and let $X$ be a complete CAT(0) space of $dim(X)<k_1+\cdots +k_r$.
Let $S_1, \ldots, S_r\subset \mathrm{Isom}(X)$ be subsets with $[s_i, s_j] = 1$ for all $s_i\in S_i$ and $s_j\in S_j$ ($i\neq j$).
If each $k_i$-element subset of $S_i$ has a fixed point in $X$ for $i = 1,\ldots, r$, then for some $i$, every finite subset of $S_i$ has a fixed point.
\end{theorem}

\begin{corollary}\label{thompson}
Thompson's group $T$ and $V$ %and ring groups 
have property $\mathrm{F}\mathcal{A}_{k}$ for semi-simple actions for every $k\in \mathbb{N}$.
\end{corollary}

Similar arguments show that Brin-Higman-Thompson groups $nV_q$ and Nekrashevych-R\"over groups $V_q(G)$ have property $\mathrm{F}\mathcal{A}_{k}$ for semi-simple actions for every $k\in \mathbb{N}$.

$T$, $V$ and their generalizations are known to have property $\mathrm{FA}$ (\cite{Farley} for $T$ and $V$, \cite{K} for $nV$).
It is also known that any isometric action of $V$ on a finite-dimensional CAT(0) cube complex fixes a point (\cite{Genevois}).
Corollary~\ref{thompson} generalizes these results.

Corollary~\ref{thompson} partially answers to a question in \cite{Genevois} (Question 1.5),
which asks whether $T$ is hyperbolically elementary and
whether every isometric action of $T$ on a
$k$-dimensional CAT(0) cube complex has a global fixed point, for every $k\geq 0$.

I would like to thank Takuya Sakasai, Javier Aramayona and Yash Lodha for helpful discussions and comments.
I would like to thank Anthony Genevois who let me know about results in \cite{Genevois} and gave many comments on this paper.
This work was supported by JSPS KAKENHI Grant Number 17J07711 and the Program for Leading Graduate Schools, MEXT, Japan.

%%%%%%%%%%%%%%%%%%%%%%%%%%%%%%%%%%%%%%%%%%%%%%%%%%%%
\section{Isometries of CAT(0) spaces and a fixed point criterion}

In this section, we compile some basic facts on isometries of complete CAT(0) spaces, mainly from \cite{BH}, and then proceed to prove Theorem~\ref{MainNew}.

Let $(X,d)$ be a metric space and let $\gamma$ be an isometry of $X$. 
The {\it translation length} of $\gamma$ is the number $|\gamma| = \Inf\{d(x, \gamma(x)) \mid x \in X\}$. 
We write $\Min(\gamma)=\{x\in X \mid d(x, \gamma(x))=|\gamma|\}$. 

\begin{lemma}[Proposition 6.2 of \cite{BH}]\label{Min}
Let $X$ be a CAT(0) space and $\gamma$ be an isometry of $X$.
$\Min(\gamma)$ is a closed convex subset of $X$.
\end{lemma}

An isometry $\gamma$ 
is called {\it semi-simple} if $\Min(\gamma)$ is non-empty.
A semi-simple isometry $\gamma$ is called {\it hyperbolic} if $|\gamma| >0$, and called {\it elliptic} if $|\gamma|=0$, i.e.\ $\gamma$ has a fixed point.
We say that an isometric action of a group $G$ on a metric space is {\it semi-simple} if every $g\in G$ is either elliptic or hyperbolic. 

The following lemma describes the structure of $\Min(\gamma)$ for hyperbolic isometries.  
\begin{lemma}\label{BH2}
Let $X$ be a CAT(0) space and $\gamma$ be a hyperbolic isometry of $X$. 
\begin{itemize}
\item[$(1)$] $\Min(\gamma)$ is isometric to a product $Y\times \mathbb{R}$, where $Y=Y\times \{0\}$ is a closed convex subspace of $X$.
If $\dim(X)<\infty$, then $\dim(Y)<\dim(X)$.
\item[$(2)$] Every isometry $\alpha$ that commutes with $\gamma$ leaves $\Min(\gamma)=Y \times \mathbb{R}$ invariant, 
and its restriction to $Y \times \mathbb{R}$ is of the form $(\alpha_1, \alpha_2)$, 
where $\alpha_1$ is an isometry of $Y$ and $\alpha_2$ is a translation on $\mathbb{R}$.
\end{itemize}
\end{lemma}

\begin{proof}
$(1)$ The proof of the existence of the product decomposition is in \cite{BH} (Theorem 2.14 and Theorem 6.8).
We observe that $Y$ has smaller dimension than $X$.
By a general fact on the topological dimension (\cite{M}), $\dim(Y\times \mathbb{R})=\dim(Y)+\dim(\mathbb{R})=\dim(Y)+1$.
Since $Y\times \mathbb{R}$ is a closed subspace of $X$, $\dim(Y\times \mathbb{R})\leq \dim(X)$.
It follows that $\dim(Y)<\dim(X)$.

$(2)$ The proof is in \cite{BH} (Theorem 6.8).
\end{proof}

\begin{lemma}[Propositions 6.2 and 6.9 of \cite{BH}]\label{BH3}
Let $X$ be a CAT(0) space and $\gamma$ be a hyperbolic isometry of $X$. 
\begin{itemize}
\item[$(1)$] If $C\subset X$ is non-empty, complete, convex and $\gamma$-invariant, then 
$\gamma$ is semi-simple if and only if $\gamma|_C$ is semi-simple.
\item[$(2)$] If $X$ splits as a product $X'\times X''$, for every isometry $\gamma=(\gamma', \gamma'')$ preserving the decomposition,
$\gamma$ is semi-simple if and only if $\gamma'$ and $\gamma''$ are semi-simple.
\end{itemize}
\end{lemma}

\begin{proof}[Proof of Theorem~\ref{MainNew}]
The proof proceeds by induction on $k$.
When $k=0$, the lemma is trivial. 
We assume that the lemma holds for all $l<k$.

We fix a complete CAT($0$) space $X$ of $\dim(X)=k$ and a semi-simple action of $G$ on $X$.
Assume to the contrary that $s$ is not elliptic. 
By the semi-simplicity of the action, $s$ is a hyperbolic element.

By the assumption, there are $H_k<G$ and $g_k\in G$ such that $g_k(\supp(H_k))$ and $\supp(H_k)$ are disjoint.
Let $H_k^{g_k}=\{g_k h g_k^{-1}\mid h\in H_k\}$. 
Since $\supp(H_k^{g_k})=g_k(\supp(H_k))$ and the action of $G$ on $A$ is faithful, 
$H_k^{g_k}$ and $H_k$ commute in $G$.
By Lemma \ref{BH2}, $H_k^{g_k}$ preserves $\Min(s)=Y\times \mathbb{R}$, 
and the action of $H_k^{g_k}$ on $\Min(s)$ splits into two actions:
an isometric action on $Y$ and an action on $\mathbb{R}$ by translations.

We first observe the action of $H_k^{g_k}$ on $\mathbb{R}$ by translations.
Since every homomorphism of $H_k$ on $\mathbb{R}$ is trivial, $H_k^{g_k}$ fixes every point in $\mathbb{R}$.

We next observe the action of $H_k^{g_k}$ on $Y$.
The action is semi-simple by Lemma \ref{BH3}.
Since $Y \times \{0\}$ is a closed convex subspace of $X$, $Y$ is a complete CAT($0$) space.
By Lemma \ref{BH2} $(1)$, the topological dimension of $Y$ is less than $k$.
By the inductive assumption to the action of $H_k^{g_k}$ on $Y$,
$s^{g_k}$ fixes a point in $Y$. 

It follows that $s^{g_k}$ fixes a point in $X$, which contradicts to the assumption that $s$ is not elliptic.
\end{proof}

%%%%%%%%%%%%%%%%%%%%%%%%%%%%%%%%%%%%%%%%%%%%%%%%%%%%%%%%%%%%%%%%%%%%%%%
\section{Applications for Thompson-type groups}
\subsection{Thompson's groups $F$, $T$ and $V$}

Classically, there are three types of Thompson's groups: $F$, $T$ and $V$.
We give definitions of these groups, based on $\cite{CFP}$.
A {\it standard dyadic interval}\/ is an interval of the form $[a/{2^b}, {(a+1)}/{2^b}]$ in the unit interval $[0,1]$, where $a\in \mathbb{Z}$ and $b\in \mathbb{N}$.   
A {\it dyadic rational}\/ is a rational number of the form $a/{2^b}$ in $[0,1]$.

{\it Thompson's group $F$} is a group of piecewise linear homeomorphisms of $[0, 1]$,
differentiable with derivatives of powers of 2 on finitely many standard dyadic intervals. 

{\it Thompson's group $T$} is a group of piecewise linear homeomorphisms of $S^1$, 
differentiable with derivatives of powers of 2 on finitely many standard dyadic intervals. 
We regard $S^1$ as the unit interval with identified endpoints.

{\it Thompson's group $V$} is a group of piecewise linear right-continuous bijections from $[0,1)$ to itself,
differentiable with derivatives of powers of 2 on finitely many standard dyadic intervals. 

We identify divisions of $[0,1]$ into finitely many standard dyadic intervals with finite rooted binary trees.
Under this identification, we represent elements of Thompson's groups by a pair of finite rooted binary trees with the same number of leaves.
The following figure shows generators of Thompson's groups, represented as tree pairs. 
$F$ is generated by $A$ and $B$ of Figure \ref{generators_of_T}.
$T$ is generated by $A$, $B$ and $C$. $V$ is generated by $A$, $B$, $C$ and $\pi_0$. 

\begin{figure}%[h]
\begin{center}
%WinTpicVersion4.32a
{\unitlength 0.1in%
\begin{picture}(41.0200,9.8500)(0.7100,-11.8500)%
% STR 2 0 3 0 Black White  
% 4 206 317 206 379 5 0 0 0
% $A=$
\put(2.0600,-3.7900){\makebox(0,0){\small $A=$}}
% LINE 2 0 3 0 Black White  
% 4 1780 200 1633 324 1780 200 1926 324
% 

\special{pn 8}%
\special{pa 1780 200}%
\special{pa 1633 324}%
\special{fp}%
\special{pa 1780 200}%
\special{pa 1926 324}%
\special{fp}%

% LINE 2 0 3 0 Black White  
% 4 1633 324 1487 446 1633 324 1780 446
% 

\special{pn 8}%
\special{pa 1633 324}%
\special{pa 1487 446}%
\special{fp}%
\special{pa 1633 324}%
\special{pa 1780 446}%
\special{fp}%

% STR 2 0 3 0 Black White  
% 4 1487 440 1487 502 5 0 0 0
% $1$
\put(14.8700,-5.0200){\makebox(0,0){\small $1$}}
% STR 2 0 3 0 Black White  
% 4 1780 440 1780 502 5 0 0 0
% $2$
\put(17.8000,-5.0200){\makebox(0,0){\small $2$}}
% STR 2 0 3 0 Black White  
% 4 1926 317 1926 379 5 0 0 0
% $3$
\put(19.2600,-3.7900){\makebox(0,0){\small $3$}}
% VECTOR 2 0 3 0 Black White  
% 4 1048 379 1340 379 1633 200 1633 200
% 

\special{pn 8}%
\special{pa 1048 379}%
\special{pa 1340 379}%
\special{fp}%
\special{sh 1}%
\special{pa 1340 379}%
\special{pa 1273 359}%
\special{pa 1287 379}%
\special{pa 1273 399}%
\special{pa 1340 379}%
\special{fp}%
\special{pa 1633 200}%
\special{pa 1633 200}%
\special{fp}%

% LINE 2 0 3 0 Black White  
% 8 667 200 521 324 667 200 813 324 813 324 667 446 813 324 960 446
% 

\special{pn 8}%
\special{pa 667 200}%
\special{pa 521 324}%
\special{fp}%
\special{pa 667 200}%
\special{pa 813 324}%
\special{fp}%
\special{pa 813 324}%
\special{pa 667 446}%
\special{fp}%
\special{pa 813 324}%
\special{pa 960 446}%
\special{fp}%

% STR 2 0 3 0 Black White  
% 4 960 440 960 502 5 0 0 0
% $3$
\put(9.6000,-5.0200){\makebox(0,0){\small $3$}}
% STR 2 0 3 0 Black White  
% 4 667 440 667 502 5 0 0 0
% $2$
\put(6.6700,-5.0200){\makebox(0,0){\small $2$}}
% STR 2 0 3 0 Black White  
% 4 521 317 521 379 5 0 0 0
% $1$
\put(5.2100,-3.7900){\makebox(0,0){\small $1$}}
% LINE 2 0 3 0 Black White  
% 4 4027 324 3880 446 4027 324 4173 446
% 

\special{pn 8}%
\special{pa 4027 324}%
\special{pa 3880 446}%
\special{fp}%
\special{pa 4027 324}%
\special{pa 4173 446}%
\special{fp}%

% LINE 2 0 3 0 Black White  
% 4 3880 446 3733 570 3880 446 4027 570
% 

\special{pn 8}%
\special{pa 3880 446}%
\special{pa 3733 570}%
\special{fp}%
\special{pa 3880 446}%
\special{pa 4027 570}%
\special{fp}%

% STR 2 0 3 0 Black White  
% 4 3733 563 3733 625 5 0 0 0
% $2$
\put(37.3300,-6.2500){\makebox(0,0){\small $2$}}
% STR 2 0 3 0 Black White  
% 4 4027 563 4027 625 5 0 0 0
% $3$
\put(40.2700,-6.2500){\makebox(0,0){\small $3$}}
% STR 2 0 3 0 Black White  
% 4 4173 440 4173 502 5 0 0 0
% $4$
\put(41.7300,-5.0200){\makebox(0,0){\small $4$}}
% VECTOR 2 0 3 0 Black White  
% 4 3243 379 3536 379 3829 200 3829 200
% 

\special{pn 8}%
\special{pa 3243 379}%
\special{pa 3536 379}%
\special{fp}%
\special{sh 1}%
\special{pa 3536 379}%
\special{pa 3469 359}%
\special{pa 3483 379}%
\special{pa 3469 399}%
\special{pa 3536 379}%
\special{fp}%
\special{pa 3829 200}%
\special{pa 3829 200}%
\special{fp}%

% LINE 2 0 3 0 Black White  
% 8 2951 324 2804 446 2951 324 3097 446 3097 446 2951 570 3097 446 3243 570
% 

\special{pn 8}%
\special{pa 2951 324}%
\special{pa 2804 446}%
\special{fp}%
\special{pa 2951 324}%
\special{pa 3097 446}%
\special{fp}%
\special{pa 3097 446}%
\special{pa 2951 570}%
\special{fp}%
\special{pa 3097 446}%
\special{pa 3243 570}%
\special{fp}%

% STR 2 0 3 0 Black White  
% 4 3243 563 3243 625 5 0 0 0
% $4$
\put(32.4300,-6.2500){\makebox(0,0){\small $4$}}
% STR 2 0 3 0 Black White  
% 4 2951 563 2951 625 5 0 0 0
% $3$
\put(29.5100,-6.2500){\makebox(0,0){\small $3$}}
% STR 2 0 3 0 Black White  
% 4 2804 440 2804 502 5 0 0 0
% $2$
\put(28.0400,-5.0200){\makebox(0,0){\small $2$}}
% LINE 2 0 3 0 Black White  
% 4 4027 324 3880 200 3880 200 3733 324
% 

\special{pn 8}%
\special{pa 4027 324}%
\special{pa 3880 200}%
\special{fp}%
\special{pa 3880 200}%
\special{pa 3733 324}%
\special{fp}%

% STR 2 0 3 0 Black White  
% 4 3733 317 3733 379 5 0 0 0
% $1$
\put(37.3300,-3.7900){\makebox(0,0){\small $1$}}
% LINE 2 0 3 0 Black White  
% 2 3763 324 3763 324
% 

\special{pn 8}%
\special{pa 3763 324}%
\special{pa 3763 324}%
\special{fp}%

% LINE 2 0 3 0 Black White  
% 4 2951 324 2804 200 2804 200 2658 324
% 

\special{pn 8}%
\special{pa 2951 324}%
\special{pa 2804 200}%
\special{fp}%
\special{pa 2804 200}%
\special{pa 2658 324}%
\special{fp}%

% STR 2 0 3 0 Black White  
% 4 2658 317 2658 379 5 0 0 0
% $1$
\put(26.5800,-3.7900){\makebox(0,0){\small $1$}}
% STR 2 0 3 0 Black White  
% 4 2365 317 2365 379 5 0 0 0
% $B=$
\put(23.6500,-3.7900){\makebox(0,0){\small $B=$}}
% LINE 2 0 3 0 Black White  
% 8 682 939 535 1062 682 939 828 1062 828 1062 682 1185 828 1062 975 1185
% 

\special{pn 8}%
\special{pa 682 939}%
\special{pa 535 1062}%
\special{fp}%
\special{pa 682 939}%
\special{pa 828 1062}%
\special{fp}%
\special{pa 828 1062}%
\special{pa 682 1185}%
\special{fp}%
\special{pa 828 1062}%
\special{pa 975 1185}%
\special{fp}%

% LINE 2 0 3 0 Black White  
% 8 1647 939 1502 1062 1647 939 1794 1062 1794 1062 1647 1185 1794 1062 1941 1185
% 

\special{pn 8}%
\special{pa 1647 939}%
\special{pa 1502 1062}%
\special{fp}%
\special{pa 1647 939}%
\special{pa 1794 1062}%
\special{fp}%
\special{pa 1794 1062}%
\special{pa 1647 1185}%
\special{fp}%
\special{pa 1794 1062}%
\special{pa 1941 1185}%
\special{fp}%

% STR 2 0 3 0 Black White  
% 4 228 1001 228 1062 5 0 0 0
% $C=$
\put(2.2800,-10.6200){\makebox(0,0){\small $C=$}}
% STR 2 0 3 0 Black White  
% 4 535 1050 535 1111 5 0 0 0
% $0$ 
\put(5.3500,-11.1100){\makebox(0,0){\small $0$ }}
% STR 2 0 3 0 Black White  
% 4 682 1173 682 1235 5 0 0 0
% $1$
\put(6.8200,-12.3500){\makebox(0,0){\small $1$}}
% STR 2 0 3 0 Black White  
% 4 975 1173 975 1235 5 0 0 0
% $2$
\put(9.7500,-12.3500){\makebox(0,0){\small $2$}}
% STR 2 0 3 0 Black White  
% 4 1502 1050 1502 1111 5 0 0 0
% $1$
\put(15.0200,-11.1100){\makebox(0,0){\small $1$}}
% STR 2 0 3 0 Black White  
% 4 1647 1173 1647 1235 5 0 0 0
% $2$ 
\put(16.4700,-12.3500){\makebox(0,0){\small $2$ }}
% STR 2 0 3 0 Black White  
% 4 1941 1173 1941 1235 5 0 0 0
% $0$
\put(19.4100,-12.3500){\makebox(0,0){\small $0$}}
% LINE 2 0 3 0 Black White  
% 8 2965 939 2819 1062 2965 939 3111 1062 3111 1062 2965 1185 3111 1062 3258 1185
% 

\special{pn 8}%
\special{pa 2965 939}%
\special{pa 2819 1062}%
\special{fp}%
\special{pa 2965 939}%
\special{pa 3111 1062}%
\special{fp}%
\special{pa 3111 1062}%
\special{pa 2965 1185}%
\special{fp}%
\special{pa 3111 1062}%
\special{pa 3258 1185}%
\special{fp}%

% LINE 2 0 3 0 Black White  
% 8 3843 939 3697 1062 3843 939 3990 1062 3990 1062 3843 1185 3990 1062 4137 1185
% 

\special{pn 8}%
\special{pa 3843 939}%
\special{pa 3697 1062}%
\special{fp}%
\special{pa 3843 939}%
\special{pa 3990 1062}%
\special{fp}%
\special{pa 3990 1062}%
\special{pa 3843 1185}%
\special{fp}%
\special{pa 3990 1062}%
\special{pa 4137 1185}%
\special{fp}%

% STR 2 0 3 0 Black White  
% 4 2395 1001 2395 1062 5 0 0 0
% $\pi_0=$
\put(23.9500,-10.6200){\makebox(0,0){\small $\pi_0=$}}
% STR 2 0 3 0 Black White  
% 4 2819 1050 2819 1111 5 0 0 0
% $1$ 
\put(28.1900,-11.1100){\makebox(0,0){\small $1$ }}
% STR 2 0 3 0 Black White  
% 4 2965 1173 2965 1235 5 0 0 0
% $2$
\put(29.6500,-12.3500){\makebox(0,0){\small $2$}}
% STR 2 0 3 0 Black White  
% 4 3258 1173 3258 1235 5 0 0 0
% $3$
\put(32.5800,-12.3500){\makebox(0,0){\small $3$}}
% STR 2 0 3 0 Black White  
% 4 3697 1050 3697 1111 5 0 0 0
% $2$
\put(36.9700,-11.1100){\makebox(0,0){\small $2$}}
% STR 2 0 3 0 Black White  
% 4 3843 1173 3843 1235 5 0 0 0
% $1$ 
\put(38.4300,-12.3500){\makebox(0,0){\small $1$ }}
% STR 2 0 3 0 Black White  
% 4 4137 1173 4137 1235 5 0 0 0
% $3$
\put(41.3700,-12.3500){\makebox(0,0){\small $3$}}
% VECTOR 2 0 3 0 Black White  
% 4 1048 1062 1340 1062 1633 884 1633 884
% 

\special{pn 8}%
\special{pa 1048 1062}%
\special{pa 1340 1062}%
\special{fp}%
\special{sh 1}%
\special{pa 1340 1062}%
\special{pa 1273 1042}%
\special{pa 1287 1062}%
\special{pa 1273 1082}%
\special{pa 1340 1062}%
\special{fp}%
\special{pa 1633 884}%
\special{pa 1633 884}%
\special{fp}%

% VECTOR 2 0 3 0 Black White  
% 4 3243 1062 3536 1062 3829 884 3829 884
% 

\special{pn 8}%
\special{pa 3243 1062}%
\special{pa 3536 1062}%
\special{fp}%
\special{sh 1}%
\special{pa 3536 1062}%
\special{pa 3469 1042}%
\special{pa 3483 1062}%
\special{pa 3469 1082}%
\special{pa 3536 1062}%
\special{fp}%
\special{pa 3829 884}%
\special{pa 3829 884}%
\special{fp}%

\end{picture}}
\caption{Generators of Thompson's groups}
\label{generators_of_T}
\end{center}
\end{figure}
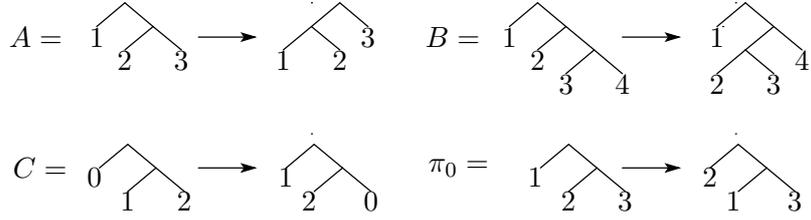

\begin{remark}\label{facts}(cf. \cite{CFP})
The commutator subgroup $[F,F]$ of $F$ is simple, and is equal to the subgroup consisting of elements which restrict to the identity on some neighborhoods of $0$ and $1$. 

Let $G$ be $F$ or $T$.
For every standard dyadic interval $I$ properly contained in $[0,1]$, let $G^I=\{g\in G\mid \supp(g)\subset I\}<G$.
Considering the canonical identification of $I$ with $[0,1]$, 
we may observe that every $G^I$ is isomorphic to $F$.
Its commutator subgroup $[G^I, G^I]$ coincides with the subgroup consisting of elements which restrict to the identity on some neighborhoods of two boundary points of $I$. 
Similarly, we may observe that $V^I$ is isomorphic to $V$, which is simple. 
\end{remark}

\begin{lemma}\label{EP} 

\begin{itemize}
\item[$(1)$] Let $G$ be either $T$ or $V$.
Let $x$ be a dyadic rational in $S^1=[0,1]/\{0=1\}$.
For every nonempty open subset $U_1$ of $S^1$
and every compact subset $U_2$ of $S^1-\{x\}$,
there is $g\in G$ such that $g(U_2)$ is contained in $U_1$.

\item[$(2)$] Let $G$ be either $F$, $T$ or $V$.
Let $I$ be a a standard dyadic interval in $[0,1]$.
For every nonempty open subset $U_1$ of $I$
and every compact subset $U_2$ in the interior of $I$,
there is $g$ in the commutator subgroup of $H^I=\{h\in G\mid \supp(h)\subset I\}$ such that $g(U_2)$ is contained in $U_1$.
\end{itemize}
\end{lemma}

\begin{proof}
$(1)$ Since $V$ contains $T$ as its subset, it is enough to show $(1)$ for $G=T$.
$U_2$ is included in the union of finitely many standard dyadic intervals. 
It is immediate from definitions that there is $g\in T$ which maps these standard dyadic intervals into a standard dyadic interval included in $U_1$.

$(2)$ There is a closed interval $J$ in the interior of $I$, containing $U_2$.
It is immediate from definitions that there is $g\in H^{J}=\{h\in G\mid \supp(h)\subset J\}$, satisfying required conditions.
When $G=F$ or $T$ , the commutator subgroup $[H^I, H^I]$ coincides with the subgroup consisting of elements which restrict to the identity on some neighborhoods of two boundary points of $I$, according to Remark~\ref{facts}.
When $G=V$, $[H^I, H^I]=H^I$.
In both cases, $H^{J}\subset [H^I, H^I]$ and thus $g$ is included in $[H^I, H^I]$.
\end{proof}

For $\varepsilon>0$, we say that a subset $U$ of $[0,1]$ is {\it of size less than }$\varepsilon$
if $U$ is in the union of finitely many standard dyadic intervals whose sum of lengths is less than $\varepsilon$. 

\begin{lemma}[cf.\ \cite{Brin}]\label{size}
For every $\varepsilon>0$,
$V$ and $T$ are generated by elements with supports of size less than $\varepsilon$.
\end{lemma}

\begin{proof}
The proof for $V$ is in \cite{Brin}.
We show the lemma for $T$. 

We take three affine copies of $[0,1]$: $U_1=[0,1/2]$, $U_2=[1/4, 3/4]$ and $U_3=[1/2, 1]$.
For every $i=1,2,3$, $F_i=\{f\in F\mid \supp(f)\subset U_i\}$ is isomorphic to $F$.
$F$ is generated by $F_1$, $F_2$ and $F_3$ (cf.\ Figure~\ref{generators_of_T} and Figure~\ref{gF}). 

We let $U_4=[3/4,1]\cup [1,1/4]$. $U_4$ is an affine copy of $[0,1]$ in $S^1=[0,1]/\{0=1\}$.
Let $T_i=\{t\in T\mid \supp(t)\subset U_i\}$ for $1\leq i \leq 4$, all isomorphic to $F$.
$T_i=F_i$ for $i=1$, $2$, $3$.
$T$ is generated by $T_1$, $T_2$, $T_3$ and $T_4$ (cf.\ Figure~\ref{generators_of_T}, Figure~\ref{gF} and Figure~\ref{gT}).

By the above observations, we may represent every element of $T$ as a composition of elements in $\bigcup_{i=1,2,3,4}T_i$.
We may represent every element of $T_i$ as a composition of elements in the copy of $\bigcup_{i=1,2,3}F_i$ in $T_i\cong F$.
Repeating this process, we may represent $g$ as a composition of elements with size less than $\varepsilon$.
\end{proof}

\begin{figure}%[h]
\begin{center}
%WinTpicVersion4.32a
{\unitlength 0.1in%
\begin{picture}(38.3700,10.0400)(2.7000,-12.0400)%
% LINE 2 0 3 0 Black White  
% 16 3787 200 3655 514 3787 200 3919 514 3919 514 3853 829 3919 514 3985 829 3721 829 3655 514 3655 514 3590 829 3590 829 3524 1143 3590 829 3655 1143
% 

\special{pn 8}%
\special{pa 3787 200}%
\special{pa 3655 514}%
\special{fp}%
\special{pa 3787 200}%
\special{pa 3919 514}%
\special{fp}%
\special{pa 3919 514}%
\special{pa 3853 829}%
\special{fp}%
\special{pa 3919 514}%
\special{pa 3985 829}%
\special{fp}%
\special{pa 3721 829}%
\special{pa 3655 514}%
\special{fp}%
\special{pa 3655 514}%
\special{pa 3590 829}%
\special{fp}%
\special{pa 3590 829}%
\special{pa 3524 1143}%
\special{fp}%
\special{pa 3590 829}%
\special{pa 3655 1143}%
\special{fp}%

% LINE 2 0 3 0 Black White  
% 4 2667 200 2535 514 2667 200 2799 514
% 

\special{pn 8}%
\special{pa 2667 200}%
\special{pa 2535 514}%
\special{fp}%
\special{pa 2667 200}%
\special{pa 2799 514}%
\special{fp}%

% LINE 2 0 3 0 Black White  
% 12 2799 514 2733 829 2799 514 2865 829 2535 514 2469 829 2535 514 2601 829 2601 829 2535 1143 2601 829 2667 1143
% 

\special{pn 8}%
\special{pa 2799 514}%
\special{pa 2733 829}%
\special{fp}%
\special{pa 2799 514}%
\special{pa 2865 829}%
\special{fp}%
\special{pa 2535 514}%
\special{pa 2469 829}%
\special{fp}%
\special{pa 2535 514}%
\special{pa 2601 829}%
\special{fp}%
\special{pa 2601 829}%
\special{pa 2535 1143}%
\special{fp}%
\special{pa 2601 829}%
\special{pa 2667 1143}%
\special{fp}%

% LINE 2 0 3 0 Black White  
% 4 1612 200 1744 514 1612 200 1481 514
% 

\special{pn 8}%
\special{pa 1612 200}%
\special{pa 1744 514}%
\special{fp}%
\special{pa 1612 200}%
\special{pa 1481 514}%
\special{fp}%

% LINE 2 0 3 0 Black White  
% 12 1481 514 1547 829 1481 514 1415 829 1744 514 1810 829 1744 514 1678 829 1678 829 1744 1143 1678 829 1612 1143
% 

\special{pn 8}%
\special{pa 1481 514}%
\special{pa 1547 829}%
\special{fp}%
\special{pa 1481 514}%
\special{pa 1415 829}%
\special{fp}%
\special{pa 1744 514}%
\special{pa 1810 829}%
\special{fp}%
\special{pa 1744 514}%
\special{pa 1678 829}%
\special{fp}%
\special{pa 1678 829}%
\special{pa 1744 1143}%
\special{fp}%
\special{pa 1678 829}%
\special{pa 1612 1143}%
\special{fp}%

% LINE 2 0 3 0 Black White  
% 12 558 200 426 514 558 200 690 514 690 514 624 829 690 514 756 829 756 829 690 1143 756 829 822 1143
% 

\special{pn 8}%
\special{pa 558 200}%
\special{pa 426 514}%
\special{fp}%
\special{pa 558 200}%
\special{pa 690 514}%
\special{fp}%
\special{pa 690 514}%
\special{pa 624 829}%
\special{fp}%
\special{pa 690 514}%
\special{pa 756 829}%
\special{fp}%
\special{pa 756 829}%
\special{pa 690 1143}%
\special{fp}%
\special{pa 756 829}%
\special{pa 822 1143}%
\special{fp}%

% LINE 2 0 3 0 Black White  
% 4 426 514 360 829 426 514 492 829
% 

\special{pn 8}%
\special{pa 426 514}%
\special{pa 360 829}%
\special{fp}%
\special{pa 426 514}%
\special{pa 492 829}%
\special{fp}%

% VECTOR 2 0 3 0 Black White  
% 4 947 672 1210 672 1098 648 1098 648
% 

\special{pn 8}%
\special{pa 947 672}%
\special{pa 1210 672}%
\special{fp}%
\special{sh 1}%
\special{pa 1210 672}%
\special{pa 1143 652}%
\special{pa 1157 672}%
\special{pa 1143 692}%
\special{pa 1210 672}%
\special{fp}%
\special{pa 1098 648}%
\special{pa 1098 648}%
\special{fp}%

% VECTOR 2 0 3 0 Black White  
% 4 2001 672 2265 672 2153 648 2153 648
% 

\special{pn 8}%
\special{pa 2001 672}%
\special{pa 2265 672}%
\special{fp}%
\special{sh 1}%
\special{pa 2265 672}%
\special{pa 2198 652}%
\special{pa 2212 672}%
\special{pa 2198 692}%
\special{pa 2265 672}%
\special{fp}%
\special{pa 2153 648}%
\special{pa 2153 648}%
\special{fp}%

% VECTOR 2 0 3 0 Black White  
% 4 3056 672 3319 672 3207 648 3207 648
% 

\special{pn 8}%
\special{pa 3056 672}%
\special{pa 3319 672}%
\special{fp}%
\special{sh 1}%
\special{pa 3319 672}%
\special{pa 3252 652}%
\special{pa 3266 672}%
\special{pa 3252 692}%
\special{pa 3319 672}%
\special{fp}%
\special{pa 3207 648}%
\special{pa 3207 648}%
\special{fp}%

% STR 2 0 3 0 Black White  
% 4 3484 1190 3484 1269 5 0 0 0
% $0$
\put(34.8400,-12.6900){\makebox(0,0){$0$}}
% STR 2 0 3 0 Black White  
% 4 3655 1190 3655 1269 5 0 0 0
% $1$
\put(36.5500,-12.6900){\makebox(0,0){$1$}}
% STR 2 0 3 0 Black White  
% 4 3715 868 3715 947 5 0 0 0
% $2$
\put(37.1500,-9.4700){\makebox(0,0){$2$}}
% STR 2 0 3 0 Black White  
% 4 3866 868 3866 947 5 0 0 0
% $3$
\put(38.6600,-9.4700){\makebox(0,0){$3$}}
% STR 2 0 3 0 Black White  
% 4 4031 868 4031 947 5 0 0 0
% $4$
\put(40.3100,-9.4700){\makebox(0,0){$4$}}
% STR 2 0 3 0 Black White  
% 4 2469 868 2469 947 5 0 0 0
% $0$
\put(24.6900,-9.4700){\makebox(0,0){$0$}}
% STR 2 0 3 0 Black White  
% 4 2522 1182 2522 1261 5 0 0 0
% $1$
\put(25.2200,-12.6100){\makebox(0,0){$1$}}
% STR 2 0 3 0 Black White  
% 4 2700 1182 2700 1261 5 0 0 0
% $2$
\put(27.0000,-12.6100){\makebox(0,0){$2$}}
% STR 2 0 3 0 Black White  
% 4 2726 868 2726 947 5 0 0 0
% $3$
\put(27.2600,-9.4700){\makebox(0,0){$3$}}
% STR 2 0 3 0 Black White  
% 4 2878 868 2878 947 5 0 0 0
% $4$
\put(28.7800,-9.4700){\makebox(0,0){$4$}}
% STR 2 0 3 0 Black White  
% 4 1415 868 1415 947 5 0 0 0
% $0$
\put(14.1500,-9.4700){\makebox(0,0){$0$}}
% STR 2 0 3 0 Black White  
% 4 1560 868 1560 947 5 0 0 0
% $1$
\put(15.6000,-9.4700){\makebox(0,0){$1$}}
% STR 2 0 3 0 Black White  
% 4 1612 1182 1612 1261 5 0 0 0
% $2$
\put(16.1200,-12.6100){\makebox(0,0){$2$}}
% STR 2 0 3 0 Black White  
% 4 1771 1182 1771 1261 5 0 0 0
% $3$
\put(17.7100,-12.6100){\makebox(0,0){$3$}}
% STR 2 0 3 0 Black White  
% 4 1810 868 1810 947 5 0 0 0
% $4$
\put(18.1000,-9.4700){\makebox(0,0){$4$}}
% STR 2 0 3 0 Black White  
% 4 360 868 360 947 5 0 0 0
% $0$
\put(3.6000,-9.4700){\makebox(0,0){$0$}}
% STR 2 0 3 0 Black White  
% 4 492 868 492 947 5 0 0 0
% $1$
\put(4.9200,-9.4700){\makebox(0,0){$1$}}
% STR 2 0 3 0 Black White  
% 4 624 868 624 947 5 0 0 0
% $2$
\put(6.2400,-9.4700){\makebox(0,0){$2$}}
% STR 2 0 3 0 Black White  
% 4 690 1182 690 1261 5 0 0 0
% $3$
\put(6.9000,-12.6100){\makebox(0,0){$3$}}
% STR 2 0 3 0 Black White  
% 4 835 1182 835 1261 5 0 0 0
% $4$
\put(8.3500,-12.6100){\makebox(0,0){$4$}}
% STR 2 0 3 0 Black White  
% 4 1079 436 1079 514 5 0 0 0
% $F_3$
\put(10.7900,-5.1400){\makebox(0,0){$F_3$}}
% STR 2 0 3 0 Black White  
% 4 2133 436 2133 514 5 0 0 0
% $F_2$
\put(21.3300,-5.1400){\makebox(0,0){$F_2$}}
% STR 2 0 3 0 Black White  
% 4 3188 436 3188 514 5 0 0 0
% $F_1$
\put(31.8800,-5.1400){\makebox(0,0){$F_1$}}
% STR 2 0 3 0 Black White  
% 4 4242 593 4242 672 5 0 0 0
% $=A$
\put(42.4200,-6.7200){\makebox(0,0){$=A$}}
\end{picture}}
\caption{$A\in \langle F_1, F_2, F_3\rangle$}
\label{gF}
\end{center}
\end{figure}
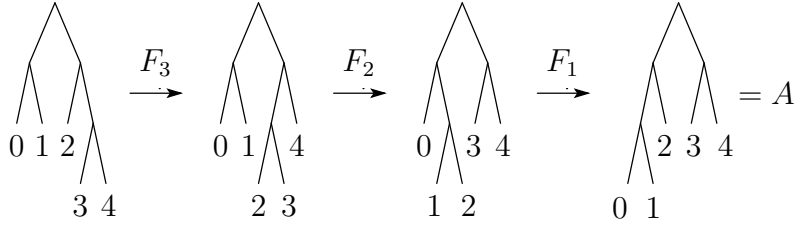

\begin{figure}%[h]
\begin{center}
\input{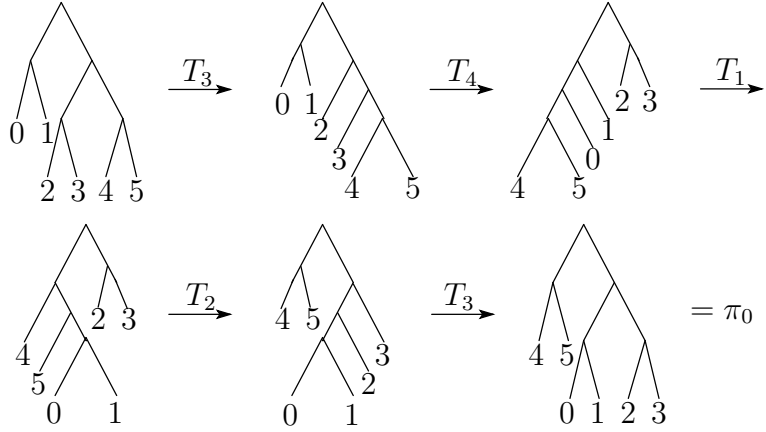}
\caption{$C\in \langle T_1, T_2, T_3, T_4\rangle$}
\label{gT}
\end{center}
\end{figure}

\begin{proof}[Proof of Corollary~\ref{thompson}]
Let $G$ be either $T$ or $V$.
Let $k\in \mathbb{N}$. 
We assume that $G$ is acting semi-simply on a $k$-dimensional complete CAT(0) space. 
According to Lemma~\ref{size}, we take a finite generating set $S$ of $G$, consisting of elements with supports of size less than $1/(k+1)$.

We fix a sequence of standard dyadic intervals $\{I_i\}_{0\leq i\leq k+1}$, where $I_{k+1}=[0,1]$ and $I_i\subset \mathrm{Int}(I_{i+1})$ for every $0\leq i\leq k$.
For every $s\in S$, there is a dyadic rational $x\in [0,1]-\supp(s)$. 
We apply Lemma~\ref{EP} $(1)$ to the interior of $I_0$ and a proper compact subset in $[0,1]-\{x\}$ including $\supp(s)$,
and take $g_0\in G$ such that $\supp(s^{g_0})=g_0(\supp(s))\subset \mathrm{Int}(I_0)$. 

For every $1\leq i\leq k$, we let $H_i$ be the commutator subgroup of $H^{I_i}=\{h\in G\mid \supp(h)\subset I_i\}$,
which is isomorphic to either $[F,F]$ or $V$.
Since both $[F,F]$ and $V$ are simple, $\{H_i\}_{1\leq i\leq k}$ satisfies the condition $(i)$ in Theorem~\ref{MainNew}.

For every $0\leq i\leq k$, there is a nonempty open subset $\bar{I_i}$ in $I_{i+1}$ such that $I_i\cap \bar{I_i}=\emptyset$.
For every $1\leq i\leq k$, we apply Lemma~\ref{EP} $(2)$ to $I=I_{i+1}$, $U_1=\bar{I_i}$ and $U_2=I_i$,
and we get $g_{i}\in H_{i+1}$ such that $g_i(I_i)\subset \bar{I_i}$.
$\{H_i\}_{1\leq i\leq k}$ and $\{g_i\}_{1\leq i\leq k}$ satisfies the condition $(ii)$ in Theorem~\ref{MainNew}.
Applying Theorem~\ref{MainNew} to a conjugate, every $s\in S$ is elliptic.

We fix mutually disjoint nonempty open subsets $J_1, \ldots, J_{k+1}$ of $[0,1]$.
For every subset $S_{k+1}\subset S$ of $(k+1)$ elements, there is a dyadic rational which is not included in $\supp(S_k)$. 
By Lemma~\ref{EP} $(1)$, there are $f_1,\ldots,f_{k+1}\in G$ such that $f_i(\supp(S_{k+1}))\subset J_i$. 
By applying Theorem~\ref{Bridson} to $S_{k+1}$, we see elements in $S_{k+1}$ have a common fixed point.
By applying Theorem~\ref{Bridson} again to $S$, we see elements in $S$ have a common fixed point, and thus there is a global fixed point for the action of $G$ on $X$.
\end{proof}

We note that various generalizations of $V$ have property $\mathrm{F}\mathcal{A}_{k}$. 
For each $q\in \mathbb{N}$ where $q\geq 2$, 
we consider standard $q$-adic intervals, instead of standard dyadic intervals, 
and define {\it Higman-Thompson groups} $V_{q}$ (\cite{Higman}).
By definition, $V_{2}$ coincides with $V$.
For each $q,n\in \mathbb{N}$ where $q\geq 2$, 
we consider products of $n$ standard $q$-adic intervals contained in $[0,1]^n$, instead of standard dyadic intervals in $[0,1]$, 
and define {\it Brin-Higman-Thompson groups} $nV_{q}$  (\cite{Brin}). 
By definition, $1V_{q}$ coincides with $V_q$.

There is a version of Lemma~\ref{size} for $nV_q$ (Proposition 3.2 of \cite{Brin}), and we can take a generating set of $nV_{q}$, consisting of elements with supports of arbitrary small size.
Since $nV_{q}$ is known to be finitely generated, we can assume that such generating set is finite.
Let $I$ be the product of $n$ standard $q$-adic intervals.
Lemma~\ref{EP} holds true for $G=nV_{q}$ and $I$.
If we define $(nV_{q})^I$ similarly as in Remark~\ref{facts},
we may observe that $(nV_{q})^I$ is isomorphic to $nV_q$, whose commutator subgroup is simple.
These observations enables us to carry on the same arguments as in the proof of Corollary~\ref{thompson}, and we get the following.

\begin{corollary}\label{genV}
For every $q,n\in \mathbb{N}$ $(q\geq 2)$, $nV_{q}$ has property $\mathrm{F}\mathcal{A}_{k}$ for every $k\in \mathbb{N}$.
\end{corollary}

Another generalization of $V$ is {\it Nekrashevych-R\"over groups} $V_q(G)$ (\cite{FH}, \cite{H}, \cite{N}, \cite{R}).
We fix $q\in \mathbb{N}$ ($q\geq 2$) and a subgroup $G$ of the symmetric group $\mathfrak{S}_q$.
Let $A_q$ be a set of alphabets $A_q=\{0,1,\ldots, q-1\}$.
We consider an action of $G$ on $C_q=A_q^{\mathbb{N}}$, which is induced by the action of $G$ on $A_q$.
Let $C_q$ be the set of all infinite words in $A_q$.
For a finite word $w$ in $A_q$, let $I(w)$ be a subset of $C_q$, consisting of all words which start with $w$.
A {\it division of $C_q$} is a sequence of finite words $(w_i)_{1\leq i\leq m}$ such that $\{I(w_i)\}_{1\leq i\leq m}$ are mutually disjoint and cover $C_q$.
Given a pair of divisions $W=(w_i)_{1\leq i\leq m}$, $W'=(w'_i)_{1\leq i\leq m}$ with the same cardinality 
and a sequence $(\sigma_i)_{1\leq i\leq m}$ of elements of $G$,
we define a bijection $v=v(W,W',(\sigma_i)_i)$ between $C_q$ by $v(w_iu)=w'_i\sigma_i(u)$, for every $i$ and every suffix $u$.
$V_q(G)$ is a group consisting of all bijections of the form $v(W,W',(\sigma_i)_i)$.

It is known that $V_q(G)$ is finitely generated and virtually simple (\cite{FH}).
By identifying $C_q$ with the unit interval and identifying subsets of the form $I(w)$ with standard $q$-adic intervals,
we may observe that $V_q(\{1\})$ coincides with $V_q$. 
Under this identification, we may carry on the same arguments as in the proof of Corollary~\ref{thompson}, and we get the following.

\begin{corollary}\label{Nekrashevych}
For every $q\in \mathbb{N}$ and every $G<\mathfrak{S}_n$, Nekrashevych-R\"over groups $V_q(G)$ has property $\mathrm{F}\mathcal{A}_{k}$ for every $k\in \mathbb{N}$.
\end{corollary}

We note that every $V_q(G)$ acts properly on an infinite-dimensional CAT(0) cube complex (\cite{H}).

%%%%%%%%%%%%%%%%%%%%%%%%%%%%%%%%%%%%%%%%%%%%%%%%%%%%%%%%%%%%%%%%%%%%%%

\end{document}